\documentclass[11pt,a4paper]{article}
\usepackage{amsmath,amsfonts,amssymb,latexsym,graphics,epsfig,url}
\usepackage{hyperref}
\usepackage{color}
\usepackage{amsthm}
\usepackage{tikz}

\usepackage[english]{babel}
\usepackage{epsfig}
\usepackage{graphicx}

\newcommand{\executeiffilenewer}[3]{%
\ifnum\pdfstrcmp{\pdffilemoddate{#1}}%
{\pdffilemoddate{#2}}>0%
{\immediate\write18{#3}}\fi%
}

\newtheorem{theo}{Theorem}[section]
\newtheorem{lema}[theo]{Lemma}

\newtheorem{propo}[theo]{Proposition}

\DeclareMathOperator{\diag}{diag}

\DeclareMathOperator{\dgr}{dgr}
\DeclareMathOperator{\tr}{tr}

\DeclareMathOperator{\spec}{sp}

\def\j{\mbox{\boldmath $j$}}
\def\p{\mbox{\boldmath $p$}}

\def\u{\mbox{\boldmath $u$}}
\def\v{\mbox{\boldmath $v$}}
\def\x{\mbox{\boldmath $x$}}

\def\vecu{\mbox{\boldmath $u$}}
\def\vecv{\mbox{\boldmath $v$}}

\def\vec0{\mbox{\boldmath $0$}}

\def\A{\mbox{\boldmath $A$}}

\def\D{\mbox{\boldmath $D$}}

\def\I{\mbox{\boldmath $I$}}
\def\J{\mbox{\boldmath $J$}}

\def\P{\mbox{\boldmath $P$}}

\def\R{\mbox{\boldmath $R$}}
\def\U{\mbox{\boldmath $U$}}

\def\OO{\mbox{\boldmath $\Omega$}}

\def\G{\Gamma}
\def\Re{\mathbb R}

\begin{document}
\title{Equivalent characterizations of the spectra of graphs and applications
to measures of distance-regularity}

\author{V. Diego, J. F\`abrega, and M.A. Fiol
\\ \\
{\small Universitat Polit\`ecnica de Catalunya} \\
{\small Dept. de Matem\`atica Aplicada IV, Barcelona, Catalonia}\\
{\small {\tt \{jfabrega,fiol,victor.diego\}@upc.edu}} \\
 }

\maketitle

\begin{abstract}
As it is well known, the spectrum $\spec \G$ (of the adjacency matrix $\A$) of a graph $\G$, with $d$ distinct eigenvalues other than its spectral radius $\lambda_0$, usually provides a lot of information about the structure of $\G$.
Moreover, from $\spec \G$ we can define the so-called predistance polynomials
$p_0,\ldots,p_d\in \Re_d[x]$, with $\dgr p_i=i$, $i=0,\ldots,d$,  which are orthogonal with respect to the scalar product $\langle f, g\rangle_{\G} =\frac{1}{n}\tr (f(\A)g(\A))$ and normalized in such a way that
$\|p_i\|_{\G}^2=p_i(\lambda_0)$. They can be seen as a generalization for any graph of the distance polynomials of a distance-regular graph. Going further,
 we consider the preintersection numbers $\xi_{ij}^h$ for  $i,j,h\in\{0,\ldots,d\}$, which generalize the intersection numbers of a distance-regular graph, and they are the Fourier
coefficients of $p_ip_j$ in terms of the basis $\{p_h\}_{0\le h\le d}$. The aim of this paper is to show that, for any graph $\G$, the information contained in its spectrum, predistance polynomials, and preintersection numbers is equivalent. Also, we give some characterizations of distance-regularity which are based on the above concepts. For instance, we comment upon the so-called spectral excess theorem stating that a connected regular graph $G$ is distance-regular if and only if its spectral excess, which is the value of $p_d$ at $\lambda_0$, equals the average excess, that is, the mean of the numbers of vertices at extremal distance $d$ from every vertex.
\end{abstract}





\section{Preliminaries}

As it is well known, two main concepts involved in the study of a distance-regular graph $\G$ are the intersection parameters and the distance polynomials. The former gives information about the combinatorial structure of $\G$, whereas the latter constitute an orthogonal sequence and yields the distance matrices of $\G$.
Moreover, both informations are univocally determined by the spectrum of $\G$ (that is, by the adjacency eigenvalues and their multiplicities).
Both concepts were generalized for any graph (see Fiol and Garriga \cite{fg97}) and, hence, they were called
preintersection numbers and predistance polynomials. In this more general framework, it happens that some basic
properties of the intersection numbers and the distance polynomials still hold. For instance, the preintersection numbers are somewhat related with the combinatorial properties of the graph, and the predistance polynomials are also an orthogonal sequence having similar properties as the ones that inspired them.

The main concern of this paper is to show that, for any graph $G$, the information contained in its spectrum, preintersection polynomials, and preintersection numbers is equivalent.
To do our task, we use both algebraic and combinatorial techniques.
As an application of our results, we provide news
characterizations of distance-regularity which are based on the above concepts.
For instance
the so-called {\em spectral excess theorem}, stating that
a connected regular graph $G$ is distance-regular if and only if its spectral excess equals the average excess, is revisited.

We first recall some basic concepts, notation, and results on which
our study is based. For more background on spectra of
graphs, distance-regular graphs, and their characterizations,
see \cite{b93,bcn89,bh12,cds82,vdkt12,f02,g93}.
Throughout the paper, $\G=(V,E)$ stands for a (simple and finite) connected graph with vertex set $V$ and edge set $E$. We denote by $n$ the number of vertices and by $e$ the number of edges. Adjacency between vertices $u$ and $v$ ($uv\in E$) will be denoted by $u\sim v$.
The \emph{adjacency matrix} $\A$ of $\G$ is the $01$-matrix with rows and columns indexed by the vertices, such that $(\A)_{uv}=1$ if and only if $u\sim v$.

\section{Three equivalent pieces of information}
\label{3pieces}

\subsection{The spectrum}
One of the most important tools in the study of the algebraic properties of a graph $\Gamma$ is its spectrum. The {\em spectrum} of $\G$ is the set of eigenvalues of its adjacency matrix $\A$ togheter with their multiplicities:
\begin{equation}
\label{spec}
\spec \G  = \{\lambda_0^{m_0},\lambda_1^{m_1},\ldots,\lambda_d^{m_d}\},
\end{equation}
where  $\lambda_0>\lambda_1>\cdots >\lambda_d$ and, for $i=0,\ldots,d$, the superscript $m_i$
stand for the  multiplicity of the eigenvalue  $\lambda_i$. Notice that, since $\G$ is connected, $m_0=1$, and if $\G$ is $k$-regular, then $\lambda_0=k$. Throughout the paper, $d$ will always denote the number of distinct eigenvalues minus one.

\subsection{The predistance polynomials}
\label{sec-predist}

Given a graph $\Gamma$ with spectrum as above, the {\em predistance polynomials}
$p_0,\ldots,p_d$, introduced by Fiol and Garriga in \cite{fg97},
are   polynomials in $\Re_d[x]$, with $\dgr p_i=i$,  which are orthogonal with respect to the scalar product
\begin{equation}
\label{produc}
\langle f, g\rangle_{\G} =\frac{1}{n}\tr (f(\A)g(\A))=\frac{1}{n} \sum_{i=0}^d m_i
f(\lambda_i) g(\lambda_i),
\end{equation}
and normalized in such a way that
$\|p_i\|_{\G}^2=p_i(\lambda_0)$ (this always makes sense since it is known that $p_i(\lambda_0)>0$ for every $i=0,\ldots,d$\,).
Some basic properties of these polynomials, which can be seen as a generalization of the distance polynomials of a distance-regular graph, are given in the following lemma, see \cite{cffg09}.
\begin{lema}
\label{ortho-pol}
Let $\G$ be a graph with average degree $\overline{k}=2e/n$, predistance polynomials $p_i$, and consider their sums $q_i=p_0+\cdots+p_i$, for $i=0,\ldots,d$. Then,
\begin{itemize}
\item[$(a)$]
$p_0=1$, $p_1=(\lambda_0/\overline{k})x$, and the constants of the three-term recurrence
\begin{equation}
\label{3-term-recur}
xp_i=\beta_{i-1}p_{i-1}+\alpha_i p_i+\gamma_{i+1}p_{i+1},
\end{equation}
where $\beta_{-1}=\gamma_{d+1}=0$, satisfy:
\begin{itemize}
\item[$(a1)$]
$\alpha_i+\beta_i+\gamma_i=\lambda_0$, for $i=0,\ldots,d$;
\item[$(a2)$]
$p_{i-1}(\lambda_0)\beta_{i-1}=p_i(\lambda_0)\gamma_i$, for $i=1,\ldots,d$.
\end{itemize}
\item[$(b)$]
$\displaystyle p_d(\lambda_0)=n\left(\sum_{i=0}^d \frac{\pi_0^2}{m_i\pi_i^2}\right)^{-1}$, where $\displaystyle \pi_i=\prod_{j\neq i}|\lambda_i-\lambda_j|$, for $i=0,\ldots,d$.
\item[$(c)$]
$1=q_0(\lambda_0)<q_1(\lambda_0) <\cdots <q_d(\lambda_0)=n$, and $q_d(\lambda_i)=0$ for every $i\neq 0$. Thus, $q_d=H$ is
the Hoffman polynomial characterizing the regularity of
$\G$ by the condition $H(\A)=\J$, where $\J$ stands for the  all-$1$ matrix $($see
Hoffman~\cite{hof63}\,$)$.
\item [$(d)$]
The three-term recurrence \eqref{3-term-recur} can be represented through a tridiagonal $(d+1)\times (d+1)$ matrix $\R$ such that, in the quotient ring $\Re[x]/(m)$, where $(m)$ is the ideal generated by the minimal polynomial $m=\prod_{i=0}^d(x-\lambda_i)$ of $\A$, it satisfies
\begin{equation}
\label{3-term-matrix}
x\p=x\left(\begin{array} {c} p_0\\p_1\\p_2\\
\vdots\\p_d\end{array}\right)= \left(\begin{array}
{ccccc}
\alpha_0  & \gamma_1    &            &               &              \\
\beta_0   & \alpha_1    & \gamma_2   &               &              \\
          & \beta_1     & \alpha_2   &               &              \\
          &             &            & \ddots        & \gamma_d     \\
          &             &            &  \beta_{d-1}  & \alpha_d     \\
\end{array}\right)
 \left(\begin{array} {c} p_0\\p_1\\p_2\\
\vdots\\p_d\end{array}\right)=\R\p\,.
\end{equation}
\end{itemize}
\end{lema}

\subsection{The preintersection numbers}
The preintersection numbers can be seen as a generalization of the intersection numbers of a distance-regular graph, which are closely related to its combinatorial properties (see e.g. Biggs \cite{b93}). In the more general context of any graph, the preintersection numbers give us an algebraic information on the graph, which is of the same nature as the spectrum of its adjacency matrix.
More precisely, the {\em preintersection numbers} $\xi_{ij}^h$, $i,j,h\in\{0,\ldots,d\}$, are the Fourier
coefficients of $p_ip_j$ in terms of the basis $\{p_h\}_{0\le h\le d}$, that is,
\begin{equation}
\label{preintersec}
\xi_{ij}^h=\frac{\langle
p_ip_j,p_h\rangle_{\G}}{\|p_h\|_{\G}^2}=\frac{1}{np_h(\lambda_0)}\sum_{r=0}^d
m(\lambda_r) p_i(\lambda_r)p_j(\lambda_r)p_h(\lambda_r).
\end{equation}
Notice that, in particular, the coefficients of the three-term
recurrence \eqref{3-term-recur} are $\alpha_i = \xi_{1,i}^i$,
$\beta_i = \xi_{1,i+1}^i$, and $\gamma_i = \xi_{1,i-1}^i$. In fact, from our derivations will be clear that such coefficients determine all the other preintersection numbers.

\section{Formulas and procedures for equivalence}
In this section we study the equivalence between the three pieces of information described in Section~\ref{3pieces}. Namely, the spectrum, the predistance polynomials, and the preintersection numbers of a given graph.

\subsection{From the spectrum to the predistance polynomials}
\label{Sb3.1}
As mentioned above, the spectrum of a graph plays a central role in the study of its algebraic and combinatorial properties. To obtain the predistance polynomials introduced in Subsection~\ref{sec-predist} we consider the scalar product defined in (\ref{produc}) and apply the Gram-Schmidt orthogonalization method to the basis $\{1,x,\ldots, x^d\}$, normalizing the obtained sequence of orthogonal polynomials in such a way that $||p_i||^2=p_i(\lambda_0)$. (This makes sense since, from the theory of orthogonal polynomials, it is known that $p_i(\lambda_0)>0$ for any $i=0,\ldots,d$.) As mentioned in Lemma~\ref{ortho-pol}, $H=p_0+\cdots+p_d$ is the Hoffman polynomial satisfying $H(\lambda_i)=0$ for $i>0$, $H(\lambda_0)=n$, and characterizing the regularity of the graph by the condition $H(\A)=\J$.

\subsection{From the predistance polynomials to the spectrum}
\label{Sb3.2}
\label{pol->sp}
In this subsection we show how the spectrum of a graph $\Gamma$ can be obtained from its predistance polynomials.

\begin{propo}
\label{poly-sp}
Let $p_0,p_1,\ldots,p_d$ be the predistance polynomials of a graph $\Gamma$, with $\omega_i^j$ being the coefficient of $x^j$ in $p_i$.
Then,
\begin{enumerate}
\item[$(a)$]
The different eigenvalues $\lambda_i\neq \lambda_0$ of $\Gamma$ are the $d$ distinct zeros of the Hoffman polynomial $H=p_0+p_1+\cdots+p_d$.
\item[$(b)$]
The largest eigenvalue (spectral radius) is
\begin{equation}
\label{lambda0}
\textstyle\lambda_0=-\frac{\omega_1^1\omega_2^0}{\omega_2^2}.
\end{equation}
\item[$(c)$]
The multiplicity of each eigenvalue $\lambda_i$, for $i=0,\ldots, d$, is
\begin{equation}
\label{mults}
m_i=n\left( \sum_{j=0}^d\frac{p_j(\lambda_i)^2}{p_j(\lambda_0)}\right)^{-1}.
\end{equation}
\end{enumerate}
\end{propo}
\begin{proof}
$(a)$ As mentioned in Lemma~\ref{ortho-pol}, $H=p_0+\cdots+p_d$ is the Hoffman polynomial satisfying $H(\lambda_i)=0$ for $i=1,\ldots,d$.

$(b)$ The expressions for $p_0$ and $p_1$ (see Lemma \ref{ortho-pol}$(a)$) imply that $\omega_0^0=1$
and $\omega_1^0=0$. Then,
\begin{equation}
\label{small-values}
\alpha_0=0,\qquad \alpha_1=-\frac{\omega_2^1}{\omega_2^2}, \qquad\mbox{and}\qquad \beta_0=-\frac{\omega_1^1}{\omega_2^2}\omega_2^0,
\end{equation}
and \eqref{lambda0} follows from $\lambda_0=\alpha_0+\beta_0$.

$(c)$ This is a result from \cite{f02} (see the proof of Proposition \ref{bbb}.)


\end{proof}

The spectral radius can also be determined as the largest root of the polynomial
\begin{equation}
\label{det_l0}
h=\left(\sum_{i=1}^{d}\frac{\lambda_i}{\,p_d(\lambda_i)}\prod_{\substack{j=1\\ j\ne i}}^d\frac{x-\lambda_j}{\lambda_i-\lambda_j}\right)p_d(x)-x.
\end{equation}
This comes from the combination of the following two facts:
the multiplicity of each eigenvalue can be also calculated as
\begin{equation}
\label{multi}
m_i=\frac{\phi_0\,p_d(\lambda_0)}{\phi_i\,p_d(\lambda_i)},\quad  \textrm{for } i=0,\ldots, d,
\end{equation}
where $\phi_i=\prod_{j=0, j\ne i}^{d}(\lambda_0-\lambda_j)$, see \cite{f02},
and the sum of all the eigenvalues has to be zero,
$\sum_{i=0}^d m_i\lambda_i=\tr \A=0$.
Note that, in fact,  the polynomial $h$ has also the roots $\lambda_1,\ldots,\lambda_d$.

Another approach is to notice that each coefficient of $H(x)=\sum_{j=0}^d h_j x^j$ can be written
as $h_j=w_j^j+w_{j+1}^j+\cdots +w_{d}^j$, where $\omega_i^j$ is the coefficient of degree $j$ of the polynomial $p_i$. In particular, if $\Gamma$ is regular, then
$$
H(x)=\frac{n}{\pi_0}\prod_{i=1}^{d}(x-\lambda_i)=\frac{n}{\pi_0}\sum_{C\subset [d]}(-1)^{|C|}x^{d-|C|}\left( \prod_{j\in C}\lambda_j\right),
$$
where $[d]=\{0,1,\ldots,d\}$,  $\pi_0=\prod_{j=1}^{d}|\lambda_0-\lambda_j|$, and, hence, we have the system of $d$ equations
$$
h_j=w_j^j+w_{j+1}^j+\cdots +w^j_{d}=\frac{n}{\pi_0}\sum_{|C|=d-j}(-1)^{d-j}\left( \prod_{i\in C}{\lambda_i} \right)\qquad j=0,\ldots, d-1,
$$
with unknowns $\lambda_1,\ldots,\lambda_d$.

\subsection{From the predistance polynomials to the preintersection numbers}
\label{Sb3.4}
In this subsection we assume that the predistance polynomials of a graph $\G$ are given and, from them, we want to obtain its preintersection numbers. Of course, we could do so by applying \eqref{preintersec}, but this requires to know the spectrum of $\G$, which requires an intermediate computation (as shown in Subsection \ref{pol->sp}).
Consequently, we want to relate directly the preintersection numbers to (the coefficients of) the predistance polynomials. With this aim, we use both the three-term recurrence \eqref{3-term-recur} and the generic expression of each polynomial $p_i$ as above.
This leads to the following result.
\begin{propo}
\label{propo-pol-pre}
Given the predistance polynomials of a graph $\G$, $p_i=\sum_{j=0}^i \omega_i^j x^j$, its preintersection numbers are:
\begin{itemize}
\item[$(a)$]
$\alpha_0=-\frac{\omega_{1}^{0}}{\omega_{1}^{1}}$, \qquad $\alpha_i=\frac{\omega_{i}^{i-1}}{\omega_{i}^{i}}-\frac{\omega_{i+1}^{i}}{\omega_{i+1}^{i+1}}$ \quad $(1\leq i \leq d-1)$;
\item[$(b)$]
$\beta_i=\frac{\omega_{i+1}^{i-1}}{\omega_{i}^{i}}-\frac{\omega_{i+1}^{i}}{\omega_{i}^{i}}
\left(\frac{\omega_{i+1}^{i}}{\omega_{i+1}^{i+1}}-\frac{\omega_{i+2}^{i+1}}{\omega_{i+2}^{i+2}}\right)
-\frac{\omega_{i+1}^{i+1}}{\omega_{i+2}^{i+2}}\frac{\omega_{i+2}^{i}}{\omega_{i}^{i}}$\quad $(0\leq i \leq d-2)$;
\item[$(c)$]
$\gamma_i=\frac{\omega_{i-1}^{i-1}}{\omega_{i}^{i}}$\quad $(1\leq i \leq d)$.
\end{itemize}
\end{propo}
\begin{proof}
By using the expressions of $p_{i-1}$, $p_i$, and $p_{i+1}$ in \eqref{3-term-recur}, and considering the terms of degree $i+1$, we get
$$
\omega_{i}^{i}=\gamma_{i+1}\omega_{i+1}^{i+1},\qquad i=0,\ldots, d-1,
$$
giving $(c)$.

Analogously, from the term of degree $i$, we have
$$
\omega_{i}^{i-1}=\alpha_i\omega_{i}^{i}+\gamma_{i+1}\omega_{i+1}^{i}
$$
whence, by using the value of $\gamma_{i+1}$, we obtain
$$
\omega_{i}^{i-1}=\alpha_i\omega_{i}^{i}+\frac{\omega_{i}^{i}}{\omega_{i+1}^{i+1}}\omega_{i+1}^{i}
$$
giving $(a)$ for $1\leq i\leq d-1$. The value of $\alpha_0$ is obtained from \eqref{3-term-recur} with $i=0$ and the value of $\gamma_1$.

Finally, looking at the terms of degree $i-1$:
$$
\omega_{i}^{i-2}=\beta_{i-1}\omega_{i-1}^{i-1}+\alpha_i\omega_{i}^{i-1}+\gamma_{i+1}\omega_{i+1}^{i-1},
$$
and using the values for $\alpha_i$ and $\gamma_{i+1}$, we have
$$
\omega_{i}^{i-2}=\beta_{i-1}\omega_{i-1}^{i-1}+\left(\frac{\omega_{i}^{i-1}}{\omega_{i}^{i}}
-\frac{\omega_{i+1}^{i}}{\omega_{i+1}^{i+1}}\right)\omega_{i}^{i-1}
+\frac{\omega_{i-1}^{i-1}}{\omega_{i}^{i}}\omega_{i+1}^{i-1}.
$$
This yields the value of $\beta_i$ for $1\leq i\leq d-2$. The value of $\beta_0$ is obtained from \eqref{3-term-recur} with $i=1$, and the values of $\alpha_1$ and $\gamma_2$. This also yields $(b)$ with $i=0$, by setting $\omega_{1}^{-1}=0$.
\end{proof}

Note that, in the above result, $\alpha_d$ and $\beta_{d-1}$ do not need to be mentioned, since they are computed by using Lemma \ref{ortho-pol}$(a1)$ with  $\lambda_0=\alpha_0+\beta_0$.

\subsection*{A matrix approach}
The above computation can be also carried out by using a matrix approach.
To this end, let us consider the given matrices
\begin{equation}
\label{Omega-U}
\OO=\left(\begin{array}
{ccccc}
\omega_{0}^{0}  &             &            &               &              \\
\omega_{1}^{0}  & \omega_{1}^{1}&            &               &              \\
\omega_{2}^{0}  & \omega_{2}^{0}& \omega_{2}^{2}  &               &              \\
          \vdots    &             &            & \ddots        &           \\
 \omega_{d}^{0}&             &     \hdots       &               & \omega_{d}^{d}     \\
\end{array}\right)\qquad\mbox{and}\qquad \U=\left(\begin{array}
{ccccc}
0       &  1           & 0         & \hdots   &   0           \\
0       &  0           & 1         &          &              \\
\vdots  &              &           & \ddots   &              \\
 0      &  \hdots      &           & 0        &   1 \\
 0      &  0           & 0         & \hdots   &   0
\end{array}\right),
\end{equation}
where, as above, the $\omega_i^j$, $i,j=0,\ldots,d$ stand for the coefficients of the predistance polynomials. From them, we want to find the tridiagonal matrix of the preintersection numbers of $\Gamma$:
 $$
\R = \left(\begin{array}
{ccccc}
\alpha_0  & \gamma_1    &            &             &                \\
\beta_0   & \alpha_1    & \gamma_2   &             &                \\
          & \beta_1     & \alpha_2   &             &                \\
          &             &            & \ddots      & \gamma_d       \\
          &             &            & \beta_{d-1} & \alpha_d
\end{array}\right).
$$

Then, we have the following result.
\begin{propo}
\label{matrixapproach}
Let $\Gamma$ be a graph  with predistance polynomials $p_0,\ldots,p_d$, and coefficient matrix $\OO$. Let $\OO'$ and $\R'$ be the matrices obtained, respectively, from $\OO$ and $\R$ by removing its last row. Then,
$$
\R'=\OO' \U \OO^{-1},
$$
\end{propo}
\begin{proof}
By using the (column) vectors $\p=(p_0,p_1,\ldots,p_d)^{\top}$ and $\x=(1,x,\ldots,x^d)^{\top}$, and $\p'$ and $\x'$ obtained from $\p$ and $\x$ by deleting the last entry, we have $\p=\OO\x$, $\p'=\OO'\x'$, and $x\x'=\U\x$. Moreover, the first $d$ equations in  \eqref{3-term-matrix} are $x\p'=\R'\p$. Then, all together yields
$$
x\OO'\x'=\OO'\U\x=\R'\OO\x,
$$
so that $(\OO'\U-\R'\OO)\x=\vec0$ and, then, it must be $\OO'\U=\R'\OO$, whence the result follows.
\end{proof}
Finally, the last row of $\R$ is computed by using Lemma \ref{ortho-pol}$(a1)$.

\subsection{From the preintersection numbers to the predistance polynomials}
\label{Sb3.3}
To obtain the predistance polynomials from the preintersection numbers of a graph $\Gamma$, we just need to apply the three-term recurrence \eqref{3-term-recur} which, initialized by $p_0=1$, yields:
\begin{equation}
\label{pi-by-recur}
p_{i}=\frac{1}{\gamma_i}[ (x- \alpha_{i-1}) p_{i-1} -\beta_{i-2}p_{i-2}], \quad i=1,\ldots,d.
\end{equation}
In particular, as stated in Lemma \ref{ortho-pol}$(a)$, we get $p_1=(\lambda_0/\overline{k})x$, so that  $\Gamma$ is regular if and only if $p_1=x$.

Alternatively, we can also compute $p_i$ directly by using the principal submatrix of the recurrence matrix $\R$ in \eqref{3-term-matrix}. Namely,
$$
\R_i = \left(\begin{array}
{ccccc}
\alpha_0  & \gamma_1    &            &               &              \\
\beta_0   & \alpha_1    & \gamma_2   &               &              \\
          & \beta_1     & \alpha_2   &               &              \\
          &             &    & \ddots                 & \gamma_i     \\
          &             &            &   \beta_{i-1}    & \alpha_i     \\
\end{array}\right),\qquad i=0,1,\ldots,d.
$$
\begin{propo}
The predistance polynomial $p_i$ associated to the recurrence matrix $\R$ is
\begin{equation}
p_i=\frac{1}{\gamma_0\cdots\gamma_i}p_c(\R_{i-1}), \qquad i=1,\ldots,d,
\end{equation}
where $p_c(\R_{i-1})$ stands for the characteristic polynomial of $\R_{i-1}$.
\end{propo}
\begin{proof}
By induction.
The result holds for $i=1,2$ since, by \eqref{pi-by-recur}, we get
$$
p_1=\frac{1}{\gamma_1}(x-\alpha_0)=\frac{1}{\gamma_1}\p_c(\R_0),\qquad
p_2=\frac{1}{\gamma_1\gamma_2}[(x-\alpha_0)(x-\alpha_1)-\beta_0\gamma_1]
=\frac{1}{\gamma_1\gamma_2}p_c(\R_1).
$$
Then, we assume that the result holds for all values smaller than $i(\ge 3)$ and prove that $\det(x\I-\R_{i-1})=p_i$ developing by the last column.
\end{proof}

Also, we can obtain explicit formulas for the coefficients of the polynomials in terms of the preintersection numbers.

\begin{lema}
Given the preintersection numbers $\alpha_i$, $\beta_i$, and $\gamma_i$ of a graph $\G$, the three coefficients of the higher degree terms of its predistance polynomial $p_i=\omega_i^ix^i+\omega_i^{i-1}x^{i-1}+\cdots+\omega_i^0$, are:
\begin{itemize}
\item[$(i)$]
$\omega_i^i=\frac{1}{\gamma_1\gamma_2\cdots\gamma_i}$;
\item[$(ii)$]
$\omega_i^{i-1}=-\frac{\alpha_0+\cdots+\alpha_{i-1}}{\gamma_1\gamma_2\cdots\gamma_i}$;
\item[$(iii)$]
$\omega_i^{i-2}=\dfrac{\sum_{0\leq r<s\leq i-1} \alpha_r\alpha_s-\sum_{r=0}^{i-2} \beta_r\gamma_{r+1}}{\gamma_1\gamma_2\cdots\gamma_i}$.
\end{itemize}
\end{lema}
\begin{proof}
By using induction with the three-term recurrence \eqref{3-term-recur}, we get:
\begin{enumerate}
\item[$(i)$]
The principal coefficient of the polynomial
$p_i=\frac{1}{\gamma_i}[(x-\alpha_{i-1})p_{i-1}-\beta_{i-2}p_{i-2}]$ is the principal coefficient of $\frac{1}{\gamma_i}p_{i-1}$, that is, $\frac{1}{\gamma_i}\omega_{i-1}^{i-1}$.
\item[$(ii)$]
The second coefficient of $p_i$ can be expressed in terms of the previous coefficients as:
$$
\omega_{i}^{i-1}=\frac{\omega_{i-1}^{i-2}-\alpha_{i-1}\omega_{i-1}^{i-1}}{\gamma_i},
$$
and using the first statement we have:
$$
\omega_{i}^{i-1}=\frac{\omega_{i-1}^{i-2}}{\gamma_i}-\alpha_{i-1}\frac{1}{\gamma_2\ldots\gamma_i}.
$$
\item[$(iii)$]
For the coefficient of the third highest degree term, we get:
$$
\omega_{i}^{i-2}=\frac{\omega_{i-1}^{i-3}-\alpha_{i-1}\omega_{i-1}^{i-2}-\beta_{i-2}
\omega_{i-2}^{i-2}}{\gamma_i},
$$
which, in addition with the previous results, it can be expressed as:
$$
\omega_{i}^{i-2}=\frac{\omega_{i-1}^{i-3}}{\gamma_i}-\frac{\alpha_1\alpha_{i-1}
+\cdots+\alpha_{i-2}\alpha_{i-1}}{\gamma_2\ldots\gamma_i}
-\frac{\beta_{i-2}\gamma_{i-1}}{\gamma_2\ldots\gamma_i}.
$$
\end{enumerate}
\end{proof}
Of course, this procedure can be carried on by calculating each $\omega_i^j$ from the three-term recurrence and using the expressions of the previously computed $\omega_i^{i},\ldots,\omega_i^{j+1}$.

The above computations can be also carried out by using a matrix approach.
Indeed, they can be set as a linear system by using the matrix approach in  Proposition \ref{matrixapproach} of the previous subsection.

\subsection{From the preintersection numbers to the spectrum}
\label{Sb3.5}

Let us now see how the preintersections numbers of a graph allow us to compute its spectrum.

\begin{propo}
\label{bbb}
Given a graph $\Gamma$ with $d$ distinct eigenvalues and matrix $\R$ of preintersection numbers, its spectrum $\spec \G  = \{\lambda_0^{m_0},\lambda_1^{m_1},\ldots,\lambda_d^{m_d}\}$ can be computed in the following way:
\begin{itemize}
\item[$(a)$]
The different eigenvalues $\lambda_0>\lambda_1>\cdots >\lambda_d$ of $\Gamma$ are the eigenvalues of $\R$, that is the (distinct) zeros of its characteristic polynomial
$p_c(\R)=\det (x\I-\R)$.
\item[$(b)$]
Let $\u_i$ and $\v_i$ be the standard (with first component $1$) left and right eigenvectors corresponding to $\lambda_i$. Then, the  multiplicities are given by the formulas
\begin{equation}
\label{mult-Biggs}
m_i=\frac{n}{\langle \u_i,\v_i\rangle}\qquad i=0,\ldots,d,
\end{equation}
where $n=\langle \u_0,\v_0\rangle$ is the number of vertices of $\G$.
\end{itemize}
\end{propo}
\begin{proof}
Let $\P$ be the matrix indexed with $0,\ldots,d$, and with entries $\P_{ij}=p_i(\lambda_j)$. Then, because of \eqref{3-term-matrix}, its $i$-th column $\v_i$ is a right $\lambda_i$-eigenvector of the recurrence matrix $\R$:
$\R\P=\P\D$, where $\D=\diag(\lambda_0,\ldots,\lambda_d)$. Then, as $\P^{-1}\R=\D\P^{-1}$, the $i$-th row $\u_i$ of $\P^{-1}$ is a left $\lambda_i$-eigenvector of $\R$.
Moreover, because of the orthogonal property of the predistance polynomials with respect to
the scalar product (\ref{produc}), the inverse of the matrix
$\P$ is
$$
\P^{-1}=\frac{1}{n}\left( \begin{array}{cccc}
m_0\frac{p_0(\lambda_0)}{n_0} &
m_0\frac{p_1(\lambda_0)}{n_1} &\ldots &
m_0\frac{p_d(\lambda_0)}{n_d}          \\
m_1\frac{p_0(\lambda_1)}{n_0}
&m_1\frac{p_1(\lambda_1)}{n_1} &\ldots &
m_1\frac{p_d(\lambda_1)}{n_d}   \\
 \vdots  & \vdots  &   & \vdots          \\
m_d\frac{p_0(\lambda_d)}{n_0} &
m_d\frac{p_1(\lambda_d)}{n_1} &\ldots &
m_d\frac{p_d(\lambda_d)}{n_d}
\end{array} \right)
$$
where $n_i=p_i(\lambda_0)$. Then, from
$(\P^{-1}\P)_{ii}=1$, $0\le i\le d$, we get
\begin{equation}
\label{mults}
m_i=n\left( \sum_{j=0}^d\frac{p_j(\lambda_i)^2}{p_j(\lambda_0)}\right)^{-1}=\frac{n}{\langle \u_i,\v_i\rangle}\qquad i=0,\ldots,d,
\end{equation}
as claimed. Finally, notice that, as $m_0=1$, $n=\langle \u_0,\v_0\rangle$.
\end{proof}
Note also  that, in \eqref{mults}, the right and left eigenvectors are, respectively, \linebreak $\u_i=(p_0(\lambda_i),p_1(\lambda_i),\ldots,p_d(\lambda_i))^{\top}$, and
$\u_i=\left(\frac{p_0(\lambda_i)}{p_0(\lambda_0)},\frac{p_1(\lambda_i)}{p_1(\lambda_0)},
\ldots,\frac{p_d(\lambda_i)}{p_d(\lambda_0)}\right)$.
In the particular case when $\G$ is a distance-regular graph, an alternative proof of \eqref{mult-Biggs} without using the orthogonal polynomials was given by Biggs \cite{b93}.

\subsection{From the spectrum to the preintersection numbers}
\label{Sb3.6}
As far as we know, in the case of distance-regular graphs there were no formulas relating directly the preintersection numbers to the eigenvalues and multiplicities of a graph. Within this context,
in the Appendix of the paper by Van Dam and Haemers \cite{vdh02}, the authors wrote the following:
``In this appendix we sketch a proof of the following result: for a distance-regular graph the
spectrum determines the intersection array. This less-known but relevant result (mentioned
in the introduction) has been observed before, but it does not seem to be readily available in
the literature."
Their method consists of three steps: first, use the scalar product \eqref{produc} to find the distance polynomials, as explained in Subsection \ref{Sb3.1} (apply Gram-Schmidt orthogonalisation and the normalization condition); second, compute the distance matrices of the graph by applying the distance polynomials to its adjacency matrix; and third, calculate the intersection parameters from the distance matrices.

However, in our context of a general graph, this method does not apply, since neither the distance matrices can be obtained from the predistance polynomials, not the preintersection numbers are related to such matrices.
Instead, an alternative would be to compute the predistance polynomials as in Subsection \ref{Sb3.1}, and then calculate the preintersection numbers by applying the results of Subsection \ref{Sb3.4}. Let us see that, if we follow properly this procedure, we can obtain explicit formulas for the preintersection numbers in terms only of the information given by the spectrum.
To this end, we call into play the average numbers of closed walks as a new piece of information.
In fact, these averages also determine univocally the spectrum, in the same way as the predistance polynomials and the preintersection numbers do. These averages can be seen as a generalization of the numbers of closed $\ell$-walks in a distance-regular graph, where, for any fixed length $\ell$, they do not depend on the root vertex.

\begin{propo}
\label{propo-sp-preintersec}
Let $\G$ be a graph. Then, its preintersection numbers can be computed directly from its spectrum $\spec \G$ by using the  average number of closed walks of length $\ell$, that is,
$c(\ell)=\frac{1}{n}tr(\A^{\ell})=\frac{1}{n}\sum_{i=0}^d m_i\lambda_i^{\ell}$, $\ell=0,1,2,\ldots$, and their first values are:
\begin{align}
\alpha_0 &= 0, \qquad \beta_0 = \lambda_0, \label{i=0}\\
\gamma_1 &=\frac{c(2)}{\lambda_0},\qquad \alpha_1 =\frac{c(3)}{c(2)}, \qquad \beta_1 = \lambda_0-\alpha_1-\gamma_1, \label{i=1}\\
\gamma_2 &= \frac{\lambda_0[c(2)c(4)-c(3)^2-c(2)^3]}{c(2)[\lambda_0^2c(2)-c(3)\lambda_0-c(2)^2]}, \quad
\alpha_2 = \frac{c(2)^2c(5)-2c(2)c(3)c(4)-c(3)^3}{c(2)[c(2)c(4)-c(3)^2-c(2)^3]},\quad \beta_2=\cdots \label{i=2}
\end{align}
\end{propo}
\begin{proof}
The proof is by induction. We know that, knowing the predistance polynomials $p_0,\ldots,p_{i-1}$, $i\ge 1$, the Gram-Schmidt method yields
\begin{equation}
\label{p(i)GS}
p_i=\frac{r_i(\lambda_0)}{||r_i||^2}r_i
\end{equation}
where
$$
r_i=x^i-\sum_{j=0}^{i-1} \frac{\langle {x}^i,p_j\rangle}{||p_j||^2}
=x^i-\sum_{j=0}^{i-1} \frac{\sum_{h=0}^{d}m(\lambda_h)\lambda_h p_j(\lambda_h)}{\sum_{h=0}^{d}m(\lambda_h) p_j^2(\lambda_h)}.
$$
Then, from $p_0=1$, we obtain that $p_1=\frac{\lambda_0}{c(2)}x$, whence, applying the formulas
\begin{align}
\gamma_i &=\frac{\langle {x} {p_{i-1}},p_i\rangle}{||p_i||^2}=\frac{1}{p_i(\lambda_0)}\sum_{j=0}^{d}m(\lambda_j)\lambda_j p_i(\lambda_j)p_{i-1}(\lambda_j), \label{gamma-i}\\
\alpha_i & =\frac{\langle {x} {p_i},p_i\rangle}{||p_i||^2}=\frac{1}{p_i(\lambda_0)}\sum_{j=0}^{d}m(\lambda_j)\lambda_j p_i^2(\lambda_j),\label{alpha-i}\\
 \beta_i& =\lambda_0-\alpha_i-\gamma_i. \label{beta-i},
\end{align}
with $i=0,1$ we get \eqref{i=0} and \eqref{i=1}.
In general, if all the coefficients of the predistance polynomials $p_0,\ldots,p_{i-1}$, $i\ge 1$, are given in terms of the numbers $c(\ell)$'s, we proceed in the same way by first calculating $p_i$ and then applying the formulas \eqref{gamma-i}--\eqref{alpha-i}. This assures that the obtained preintersections parameters $\alpha_i$, $\beta_i$, and $\gamma_i$ will be expressed also in terms of the $c(\ell)$'s. For instance, the computation for $i=2$ give the results in \eqref{i=2}.
\end{proof}

%
%


\section{An example}
In this section, we illustrate the previous results with one example.
Let $\G$ be the graph 4.47 of Table 4 in the textbook of Cvetkovi\'c, Doob, and  Sachs \cite{cds82}, shown in Fig. \ref{fig:4.47}, which has $n=9$ vertices, and spectrum
$$
\spec \G  = \left\{3^1,\left(\frac{-1+\sqrt{13}}{2}\right)^2,0^3,(-1)^1,\left(\frac{-1-\sqrt{13}}{2}\right)^2\right\}.
$$
Thus, $\G$ has $d+1=5$ distinct eigenvalues.

\begin{figure}[t]
    \begin{center}
  \includegraphics[width=10cm]{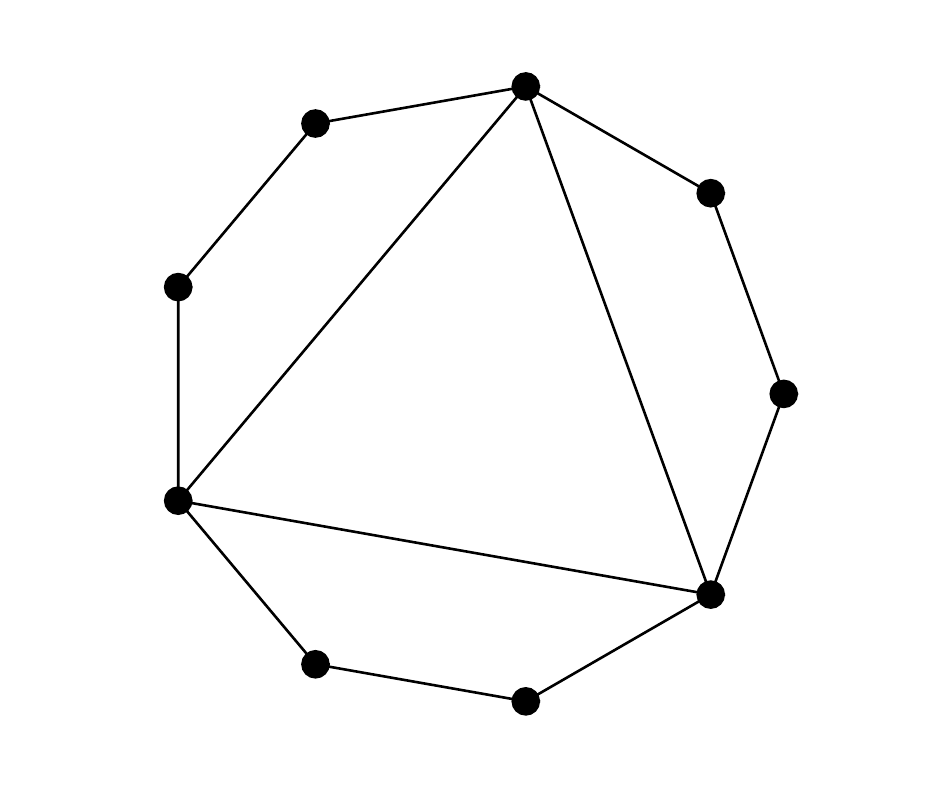}
  \end{center}
  \caption{The graph 4.47 in Table 4 of \cite{cds82}.}
  \label{fig:4.47}
\end{figure}

\subsection{From the spectrum to the predistance polynomials}
As mentioned in Subsection \ref{Sb3.1}, the sequence of predistance polynomials $p_0, p_1, p_2, p_3, p_4$  are orthogonal with respect to the scalar product
$$
\langle f, g\rangle_{\G} =\frac{1}{n}\tr (f(\A)g(\A))=\frac{1}{n} \sum_{i=0}^d m_i
f(\lambda_i) g(\lambda_i),
$$
and normalized in such a way that
$\|p_i\|_{\G}^2=p_i(\lambda_0)$. Then, we can obtain them by applying the Gram-Schmidt method, starting from the sequence $1,x,x^2,x^3,x^4$, and the result is:
$$
\begin{array}{l}
p_0(x)=1,\\[2mm]
\displaystyle p_1(x)=\frac{9}{8}x,\\[4mm]
\displaystyle p_2(x)=-\frac{268}{157}-\frac{201}{1256}x+\frac{201}{314}x^2,\\[4mm]
\displaystyle p_3(x)=\frac{23607}{50711}-\frac{83082}{50711}x-\frac{732}{2983}x^2+\frac{183}{646}x^3,\\[4mm]
\displaystyle p_4(x)=\frac{78}{323}+\frac{547}{1292}x-\frac{32}{57}x^2-\frac{113}{969}x^3+\frac{1}{12}x^4.
\end{array}
$$

\subsection{From the predistance polynomials to the spectrum}
To obtain the spectrum from the predistance polynomials, we can use the results in Proposition \ref{poly-sp}. So,
the Hoffman polynomial $H=p_0+p_1+p_2+p_3+p_4$ is
$$
H(x)=-\frac{1}{4}x-\frac{1}{6}x^2+\frac{1}{6}x^3+\frac{1}{12}x^4,
$$
with zeros being the distinct eigenvalues different from $\lambda_0$:
$$
\lambda_1=\frac{-1+\sqrt{13}}{2},\ \lambda_2=0,\ \lambda_3=-1,\ \lambda_4=\frac{-1-\sqrt{13}}{2}.
$$
Moreover, the largest root of the polynomial given by (\ref{det_l0}) is $\lambda_0=3$.
Alternatively, by \eqref{lambda0}, $\lambda_0=-\frac{\omega_1^1\omega_2^0}{\omega_2^2}=-(9/8)(-268/157)/(201/314)=3$.

Moreover, the values of the parameters $\phi_i$ and $p_d(\lambda_i)$ are:
$$
\phi_0=108,\
\phi_1=\frac{3}{2}[13-7\sqrt{13}],\
\phi_2=9,\
\phi_3=-12, \
\phi_4=\frac{3}{2}[13+7\sqrt{13}],
$$
and
$$
p_4(\lambda_0)=\frac{39}{646},\
p_4(\lambda_1)=-\frac{1}{646}[39+21\sqrt{13}],\
p_4(\lambda_2)=\frac{78}{323},
$$
$$
p_4(\lambda_3)=-\frac{351}{646},\
p_4(\lambda_4)=\frac{1}{646}[-39+21\sqrt{13}].
$$
Thus, by applying \eqref{multi}, we get the multiplicities:
$$
m_0=1,\quad
m_1=2,\quad
m_2=3,\quad
m_3=1,\quad
m_4=2.
$$

\subsection{From the predistance polynomials to the preintersection numbers}
To obtain the preintersection numbers by using the predistance polynomials, we apply  Proposition \ref{matrixapproach} giving the relationship between the preintersection matrix and the polynomial coefficient matrix of the graph.

The matrix $\OO$ containing the polynomial coefficients of the graph $\Gamma$ is:
$$
\OO = \left(\begin{array}
{ccccc}
1  & 0    &      0      &  0 &  0           \\
0   & {9}/{8}    & 0   & 0  &   0       \\
-{268}/{157}     & -{201}/{1256}     & {201}/{314}   & 0           &  0        \\
{23607}/{50711}  &   -{83082}/{50711}& -{732}/{2983} & {183}/{646} & 0        \\
78/323           & 547/1292          &-32/57         & -{113}/{969}    & 1/12
\end{array}\right).
$$
Then, with $\U$ given in \eqref{Omega-U}, Proposition \ref{matrixapproach} yields:
$$
\R'=\OO'\U{\OO}^{-1}=\left(\begin{array}
{ccccc}
0  & 8/9    &      0      &  0 &  0           \\
3   & 1/4    & 471/268   & 0  &   0       \\
0     & 67/36     &387/628   & 21641/9577  &  0        \\
0 &   0  & 6588/10519    & 27036/50711   & 1098/323       \\
\end{array}\right),
$$
and finally, we add the last row of the matrix $\R$ of preintersection numbers by using the equality $\alpha_i+\beta_i+\gamma_i=\beta_0=3$:
$$
\left(\begin{array}
{ccccc}
\alpha_0  & \gamma_1    &            &          &             \\
\beta_0   & \alpha_1    & \gamma_2   &          &             \\
          & \beta_1     & \alpha_2   & \gamma_3 &             \\
          &             & \beta_2    & \alpha_3 &             \\
          &             &            & \beta_3  &  \gamma_4   \\
          &             &            &          &  \alpha_4
\end{array}\right)=\left(\begin{array}
{ccccc}
0   & 8/9   &                   &                     &           \\
3   & 1/4   & 471/268   &                     &          \\
    & 67/36 & 387/628   & 21641/9577  &          \\
    &                   & 6588/10519  & 27036/50711 & 1098/323 \\
    &               &                   & 4082/19703   & -129/323
\end{array}\right).
$$

\subsection{From the preintersection numbers to the predistance polynomials}
 In order to obtain the predistance polynomials from the preintersection numbers of  $\Gamma$, we  apply the three-term recurrence \eqref{3-term-recur} which, initialized with $p_0=1$, yields:
$$
\begin{array}{l}
p_0(x)=1,\\[2mm]
\displaystyle p_1(x)=\frac{(x-\alpha_{0})p_{0}-\beta_{(-1)}p_{(-1)}}{\gamma_1}=\frac{9}{8}x,\\[4mm]
\displaystyle p_2(x)=\frac{(x-\frac{1}{4})p_1-3}{\frac{471}{268}}=-\frac{268}{157}-\frac{201}{1256}x+\frac{201}{314}x^2,\\[4mm]
\displaystyle p_3(x)=\frac{(x-\frac{387}{628})p_{2}-\frac{67}{36}p_{1}}{\frac{21641}{9577}}=\frac{23607}{50711}-\frac{83082}{50711}x-\frac{732}{2983}x^2+\frac{183}{646}x^3,\\[4mm]
\displaystyle p_4(x)=\frac{(x-\frac{27036}{50711})p_{3}-\frac{6588}{10519}p_{2}}{\frac{1098}{323}}=\frac{78}{323}+\frac{547}{1292}x-\frac{32}{57}x^2-\frac{113}{969}x^3+\frac{1}{12}x^4.
\end{array}
$$

Alternatively, we can compute the characteristic polynomial of each submatrix $\R_{i-1}$ for $i=1,\ldots,d$. Then, we get:
\begin{align*}
p_1(x)&=\frac{9}{8}\det\left(x\I-\left(\begin{array}{c}0\end{array}\right)\right)
=\frac{9}{8}x,\\
p_2(x)&=\frac{9}{8}\cdot \frac{268}{471}\det\left( x\I-\left(\begin{array}{cc}0 & 8/9 \\ 3 & 1/4 \\ \end{array}\right)\right)=-\frac{268}{157}-\frac{201}{1256}x+\frac{201}{314}x^2,\\
p_3(x)&=\frac{9}{8}\cdot \frac{268}{471}\cdot\frac{9577}{21641}\det \left( x\I-\left(\begin{array}{ccc}0 & 8/9 & \\ 3 & 1/4 & 471/268 \\  & 67/36 & 387/628 \\ \end{array}\right)\right) \\
&=\frac{23607}{50711}-\frac{83082}{50711}x-\frac{732}{2983}x^2+\frac{183}{646}x^3,\\
p_4(x)&=\frac{9}{8}\cdot \frac{268}{471}\cdot\frac{9577}{21641}\cdot\frac{323}{1098}\det \left( x\I-\left(\begin{array}{cccc} 0 & 8/9 &   &  \\ 3 & 1/4 & 471/268 &  \\  & 67/36 & 387/628 & 21641/9577 \\  &   & 6588/10519 & 27036/50711  \\ \end{array}\right)\right)\\
&=\frac{78}{323}+\frac{547}{1292}x-\frac{32}{57}x^2-\frac{113}{969}x^3+\frac{1}{12}x^4.
\end{align*}

We can also check that the principal coefficients of each predistance polynomials are easily determined by the parameters $\gamma_i$'s:
$$
\begin{array}{l}
\displaystyle \omega_1^1=\frac{1}{\gamma_1}=\frac{9}{8},\\[4mm]
\displaystyle \omega_2^2=\frac{1}{\gamma_1\gamma_2}=\frac{9}{8}\cdot\frac{268}{471}=\frac{201}{314},\\[4mm]
\displaystyle \omega_3^3=\frac{1}{\gamma_1\gamma_2\gamma_3}=\frac{9}{8}\cdot\frac{268}{471}\cdot\frac{9577}{21641}=\frac{183}{646},\\[4mm]
\displaystyle \omega_4^4=\frac{1}{\gamma_1\gamma_2\gamma_3\gamma_4}=\frac{9}{8}\cdot\frac{268}{471}\cdot\frac{9577}{21641}\cdot\frac{323}{1098}=\frac{1}{12}.\\[4mm]
\end{array}
$$

\subsection{From the preintersection numbers to the spectrum}
To obtain the spectrum of  $\G$,
we first compute the characteristic polynomial of the preintersection matrix
$$
\R =
\left(\begin{array}
{ccccc}
0  & 8/9   &           &               &          \\
3  & 1/4   & 471/268   &               &          \\
   & 67/36 & 387/628   & 21641/9577    &          \\
   &       & 6588/10519& 27036/50711   & 1098/323 \\
   &       &           & 4082/19703    &-129/323
\end{array}\right),
$$
which turns out to be $\phi_{\G}(x)=x^5-x^4-8 x^3+3x^2+9x$. Then, its roots are
$$
\lambda_0=3,\ \lambda_1=\frac{1}{2}(-1+\sqrt{13}),\ \lambda_2=0,\ \lambda_3=-1,\ \lambda_4=\frac{1}{2}(-1-\sqrt{13}).
$$
To compute the multiplicities, we first consider the left and right eigenvectors of $\lambda_0$:
$$
\vecu_0=\j=\left( 1,1,1,1,1\right)\qquad\mbox{and}\qquad
\vecv_0=\left( 1,\frac{27}{8},\frac{4489}{1256},\frac{100467}{101422},\frac{39}{646}\right),
$$
so giving $n=\langle \vecu_0,\vecv_0\rangle=9$. Now,
let us consider, for example, the eigenvalue $\lambda_2=0$. Then, the corresponding left and right normalized eigenvectors of $\R$ are
$$
\vecu_2=\left( 1,0,-\frac{32}{67},0,\frac{86}{183},4\right)
\qquad\mbox{and}\qquad
\vecv_2=\left( 1,0,-\frac{286}{157},\frac{23607}{50711},\frac{78}{323}\right).
$$
Then, we get
$$
m_2=\frac{n}{\langle \vecu_2,\vecv_2\rangle}=3,
$$
and similar computations give
$\langle\vecu_0,\vecv_0\rangle=9$, $\langle \vecu_1,\vecv_1\rangle=\frac{9}{2}$, $\langle\vecu_2,\vecv_2\rangle=3$, $\langle\vecu_3,\vecv_3\rangle=9$,
and  $\langle\vecu_4,\vecv_4\rangle=\frac{9}{2}$ so giving
the other multiplicities $m_1=2$, $m_3=1$, and $m_4=2$ .

\subsection{From the spectrum to the preintersection numbers}

In our case, the average numbers of walks of length $\ell=0,1,\ldots,5$ turn out to be
$$
c(0)=1,\qquad c(1)=0,\qquad c(2)=\frac{8}{3},\quad c(3)=\frac{2}{3},\qquad c(4)=16,\qquad c(5)=\frac{40}{3},
$$
and, then, Proposition
\ref{propo-sp-preintersec} gives:
$$
\alpha_0=0,\  \beta_0=3,\ \gamma_1=\frac{8}{9},\ \alpha_1=\frac{1}{4},\  \beta_1=\frac{67}{36},\  \gamma_2=\frac{471}{268},\  \alpha_2=\frac{387}{628},\ \beta_2=\frac{6588}{10519},\ldots
$$
and we can keep applying the method to obtain the remaining preintersection numbers.

\section{Some applications}

In this section we present some applications of the information given by the spectrum, the predistance polynomials, and the preintersection numbers of a given graph. Moreover we show how the equivalences of these informations allows us to rewrite some of the properties and/or conditions in different forms. We begin with some combinatorial properties of a graph that can be deduced from its preintersection numbers (see \cite{adf16}).

\subsection{Properties of the preintersection numbers}

We can see if the graph is bipartite or how large is its odd girth with simply checking at its matrix $\Omega$ of coefficients.
\begin{propo}
Let $\G$ with be a graph with $d+1$ distinct eigenvalues. Then,
\begin{itemize}
\item[$(a)$]
$\G$ is bipartite if and only if $\alpha_0=\cdots=\alpha_d=0$.
\item[$(b)$]
If $\G$ is not bipartite, then it has odd girth $2m+1$ if and only if $\alpha_0=\cdots=\alpha_{m-1}=0$ and $\alpha_m>0$.
\end{itemize}
\end{propo}

From the results in Subsection $3.3$, we observe that if $\alpha_i=0$ for $0\leq i\leq m$ then the coefficients of the preintersection polynomials $\omega_i^j$ equals $0$ if $i+j$ is odd (that is, when $i$ and $j$ have distinct parity). Then, we can rewrite this previous proposition as follows:

\begin{propo}
\begin{itemize}
\item[$(a)$]
A graph $\G$ with $d+1$ distinct eigenvalues is bipartite if and only if, in the matrix $\Omega$, $\omega_i^j=0$  for every $i+j$ odd.
\item[$(b)$]
If $\G$ is not bipartite, then it has odd girth $2m+1$ if and only if $\omega_i^j=0$ for every $i+j$ odd and $i\leq m$.
\end{itemize}
\end{propo}

Concerning the girth, we have the following results.

\begin{propo}
\begin{itemize}
\item[$(a)$]
A regular graph $\G$ has girth $2m+1$ if and only if $\alpha_0=\cdots=\alpha_{m-1}=0$, $\alpha_m\neq 0$ and $\gamma_1=\cdots=\gamma_m=1$.
\item[$(b)$]
A regular graph $\G$ has girth $2m$ if and only if $\alpha_0=\cdots=\alpha_{m-1}=0$, $\gamma_1=\cdots=\gamma_{m-1}=1$ and $\gamma_m>1$.
\end{itemize}
\end{propo}

Using the same equivalences as before, we have:

\begin{propo}
\begin{itemize}
\item[$(a)$]
A regular graph $\G$ has girth $2m+1$ if and only if
the polynomial coefficients satisfies $\omega_i^j=0$ for every $i+j$ odd and $i\leq m-1$, $\omega_m^j\neq 0$ for some $m+j$ odd, and $\omega_1^1=\cdots=\omega_m^m=1$.
\item[$(b)$]
A regular graph $\G$ has girth $2m$ if and only if
the polynomial coefficients satisfies $\omega_i^j=0$ for every $i+j$ odd and $i\leq m-1$, $\omega_1^1=\cdots=\omega_{m-1}^{m-1}=1$, and $\omega_m^m<1$.
\end{itemize}
\end{propo}


\subsection{Characterizations of distance-regularity}
Now we give some characterizations of distance-regularity in graphs, which are given in terms of the different informations considered.
We begin with the so-called `spectral excess theorem' (see Fiol and Garriga \cite{fg97}), which can be seen as a quasi-spectral characterization of a distance-regular graph.

\begin{theo}{\bf (The spectral excess theorem)}
Let $\G=(V,E)$ be a regular graph with
spectrum/predistance polynomials/preintersection numbers as above. Then $\G$ is distance-regular if an only if its spectral excess
$$
p_d(\lambda_0)=\frac{\beta_0\beta_1\cdots\beta_{d-1}}{\gamma_1\gamma_2\cdots\gamma_d}
=n\left(\sum_{i=0}^d \frac{\pi_0^2}{m_i\pi_i^2}\right)^{-1},
$$
(where $\displaystyle \pi_i=\prod_{j\neq i}|\lambda_i-\lambda_j|$, for $i=0,\ldots,d$) equals the average excess
$$
\overline{k}_d=\frac{1}{n}\sum_{u\in V} \G_d(u).
$$
\end{theo}
\begin{proof}
The result was proved in \cite{fg97} with the spectral excess $p_d(\lambda_0)$ given in terms of the spectrum. The condition involving the preintersection numbers comes from applying Lemma \ref{ortho-pol}$(a2)$ starting from $p_1(\lambda_0)=1$.
\end{proof}

The following result was proved by Abiad,Van Dam, Fiol \cite{adf16} for a more particular family of distance-regular graphs. (Here it can be shown that the conditions on the preintersection numbers $\gamma_i$'s is related to the existence of unique geodetic paths between vertices.)

\begin{theo}
\label{theo-adf16}
Let $\G$ be a graph with $d+1$ distinct eigenvalues and preintersection numbers $\gamma_i$, $i=1,\ldots,d$.
\begin{itemize}
\item[$(a)$]
If $\gamma_1=\cdots =\gamma_{d-1}=1$, then $\G$ is distance-regular.
\item[$(b)$]
If $\G$ is bipartite and $\gamma_1=\cdots =\gamma_{d-2}=1$, then $\G$ is distance-regular.
\end{itemize}
\end{theo}

This result implies that a graph is distance-regular  if its predistance polynomials are monic.

\begin{theo}
Let $\G$ be a graph with $d+1$ distinct eigenvalues and predistance polynomials $p_i$, $i=0,1,\ldots,d$.
\begin{itemize}
\item[$(a)$]
If all the $p_i$'s, are monic ($\omega_i^i=1$) for $i=1,\ldots,d-1$, then $\G$ is distance-regular.
\item[$(b)$]
If $\G$ is bipartite and all the $p_i$'s, are monic, then $\G$ is distance-regular.
\end{itemize}
\end{theo}
\begin{proof}
Apply Theorem \ref{theo-adf16} and Proposition \ref{propo-pol-pre}$(c)$ recursively from $\omega_0^0=1$.
\end{proof}

\vskip 1cm
\noindent{\large \bf Acknowledgments.}
This research is supported by the
{\em Ministerio de Ciencia e Innovaci\'on}, Spain, and the {\em European Regional
Development Fund} under project MTM2014-60127-P), and the {\em Catalan Research
Council} under project 2014SGR1147.

\end{document}